\documentclass[12pt,oneside,reqno]{amsart}
\usepackage{amssymb}
\usepackage{amsmath}
\usepackage{amsthm}
\usepackage{longtable}
\usepackage{xspace}
\usepackage{anysize} 

\allowdisplaybreaks

\theoremstyle{plain}
\newtheorem{Theorem}{Theorem}[section]
\newtheorem{Corollary}[Theorem]{Corollary}
\newtheorem{Proposition}[Theorem]{Proposition}
\newtheorem{Lemma}[Theorem]{Lemma}
\newtheorem{Def}[Theorem]{Definition}
\newtheorem*{Remark}{Remark}
\newtheorem*{Theorem*}{Theorem}
\newtheorem*{Proposition*}{Proposition}
\newtheorem*{Corollary*}{Corollary}

\newcommand{\C}{\mathbb{C}}
\newcommand{\F}{\mathbb{F}}
\newcommand{\N}{\mathbb{N}}
\newcommand{\Q}{\mathbb{Q}}

\newcommand{\Z}{\mathbb{Z}}

\newcommand{\GM}{\operatorname{\Gamma_{\mathrm{M}}}}
\newcommand{\LD}{\operatorname{Log\Gamma_{\mathrm{D}}}}
\newcommand{\LM}{\operatorname{Log\Gamma_{\mathrm{M}}}}

\newcommand{\LP}{\operatorname{Log\Gamma_{\mathit{p}}}}
\newcommand{\GC}{\operatorname{\Gamma}}

\newcommand{\LG}{\operatorname{\log\Gamma}}

\newcommand{\absp}[1]{\left|{#1}\right|_p}

\newcommand{\ol}[1]{\overline{#1}}
\newcommand{\RP}{\operatorname{\mathrm{R}_\mathit{p}}}
\newcommand{\rp}{\operatorname{\mathrm{r}_\mathit{p}}}

\newcommand{\zcp}{\mathcal{Z}_p}

\newcommand{\al}{\alpha}

\everymath{\displaystyle}

\begin{document}

\pagestyle{plain}

\title{A distribution formula for Kashio's $p$-adic log-gamma function}
\author{Eugenio Finat}
\address{Universidad de Chile, Facultad de Ciencias, Casilla 653, Santiago, Chile}
\email{e\_finat@yahoo.com}
\begin{abstract}
We study a special case of Kashio's $p$-adic $\LG$-function, that we call $\LP$,
which combines these of Morita 
and Diamond. It agrees with each of these on large parts of its domain and 
has the advantage of being a locally analytic  function. We prove a
distribution formula for $\LP$ which generalizes and links the known distribution
formulas for Diamond's and Morita's functions.
\end{abstract}
\maketitle

\section{Introduction}

Let $p$ be a fixed prime number, 
and let $\Z_p$, $\Q_p$, $\C_p$ and $\zcp$
denote, respectively, the ring of $p$-adic integers,
the field of fractions of $\Z_p$, the completion
of the algebraic closure of $\Q_p$, and 
the ring of integers of $\C_p$.  For any $x\in\zcp$,
let $\ol{x}$ be its natural image in the residue field of $\C_p$, which is isomorphic
to the algebraic closure $\ol{\F}_p$ of the finite field $\F_p$.
Also, if $x\in\zcp$ and $\ol{x}\in\F_p$ we write $\ell(x)$
for the unique natural number satisfying $1\le\ell(x)\le p$ and
$\ol{x}=\ol{\ell(x)}$. Finally, we will use the convention
 $\N_0= \N \cup\big\{0\big\}$.

In the mid 1970's, two 
$p$-adic analogues of the classical $\LG$-function 
were defined by Yasuo Morita \cite{Mor} and Jack Diamond \cite{Dia}.
Here $\LG$ is the logarithm of the classical $\GC$-function
satisfying the difference equation
\begin{equation}
\label{ClassicDiff}
\LG(x+1)-\LG(x)=\log(x) \qquad  (x>0).
\end{equation}
The function $\LG$ satisfies the  distribution formula
\begin{equation}
\label{distLG}
\sum_{k=0}^{n-1}\LG\bigg(\frac{x+k}{n}\bigg)=\LG(x)+
\frac{n-1}2\log(2\pi)+\bigg(\frac12-x\bigg)\log(n)
\quad(x>0,n\in\N).
\end{equation}
Also,  $\LG$ is the unique convex function defined on $(0,\infty)$
satisfying $\LG(1)=0$ and the difference equation \eqref{ClassicDiff}.

Morita \cite{Mor} defined a $p$-adic
analogue of $\GC$, which we will call $\GM$,
having $\Z_p$ as its domain and taking values in the units $\Z_p^*$.
For positive integers $n$, $\GM$ is defined as
$$
\GM(n):=(-1)^n\prod_{\substack{1\le j<n\\ p\nmid j}}j.
$$
Morita proved that this function is continuous on $\N$ (with the
$p$-adic topology) and extended
to a continuous function on $\Z_p$. He also showed that
$\GM:\Z_p\to\Z_p^*$
satisfies the functional equation
\begin{equation}
\label{MGammaRatio}
\frac{\GM(x+1)}{\GM(x)}=
\begin{cases}\displaystyle
  -x\quad &\text{if }x\in\Z_p^*, \\
  -1 \quad &\text{if } x\in p\Z_p.
\end{cases}
\end{equation}
Since $\GM$ is continuous on $\Z_p$, it is completely characterized by
\eqref{MGammaRatio} and by its value $\GM(1)=-1$. 

To prove analytic properties of his function, Morita actually
worked with the Iwasawa $p$-adic logarithm $\log_p$ of $\GM$
\cite[$\S$V.4.5]{Rob} \cite[$\S$45]{Sch1}.
We will write this function $\LM$, i.e.,
$$
\LM(x):=\log_p\GM(x)
\qquad\qquad(x\in\Z_p).
$$

An important property of $\LM$ is that it has a
power series expansion around $0$, valid for all $x\in p\Z_p$,
and this power series actually defines an analytic function
on the open unit ball  $B(0;1^-):=\big\{x\in\C_p\,\big|\,\absp{x}<1\big\}
\subset\C_p$ \cite[Lemma 58.2]{Sch1}. 
Hence, we can extend the domain of $\LM$ to $B(0;1^-)\cup\Z_p$.
It can be shown that with this extended definition
$\LM$ is no longer the Iwasawa 
logarithm of a $p$-adic function.

Taking the Iwasawa logarithm on both sides of \eqref{MGammaRatio},
we find that $\LM$ satisfies the difference equation
\begin{equation}
\label{MorDiff}
\LM(x+1)-\LM(x)=
\begin{cases}\displaystyle
  \log_p(x)\quad &\text{if }x\in\Z_p^*, \\
  0 \quad &\text{if } x\in p\Z_p ,
\end{cases}
\end{equation}
in analogy to \eqref{ClassicDiff}. It is
again immediate that $\LM$ is uniquely determined
by the value $\LM(1)=0$ and by the difference
equation \eqref{MorDiff}. Morita's function
satisfies the reflection formula \cite[Prop.~11.5.13.(2)]{Coh}
\begin{equation}
\label{MorRef}
\LM(1-x)+\LM(x)=0 \qquad\qquad (x\in\Z_p),
\end{equation}
and can be given by the integral formula \cite[$\S58$]{Sch1}
\begin{equation}
\label{VolkenMor}
\LM(x) =\int_{\Z_p}
(x+t)\big(\log_p(x+t)-1\big)
\chi^{\phantom{-1}}_{\Z_p^*}\!(x+t)dt
\qquad (x\in\Z_p).
\end{equation}
Here $\chi^{\phantom{-1}}_{\Z_p^*}$ denotes the characteristic
function of $\Z_p^*$, and the integral on the
right is the
Volkenborn integral: if $g:\Z_p\to\C_p$ then
$g$ is Volkenborn integrable if the limit
$$
\int_{\Z_p} g(t)dt:=
\lim_{n\to\infty}\frac1{p^n}\sum_{j=0}^{p^n-1} g(j)
$$
exists, and this value is called the Volkenborn integral of $g$  \cite[$\S$V.5]{Rob} \cite[$\S$55]{Sch1}.

Morita's function satisfies, in analogy with \eqref{distLG}, the distribution formula
\cite[Prop.~11.5.13.(4)]{Coh}
\begin{equation}
\label{distLM}
\sum_{0\le j<n}\LM\bigg(\frac{x+j}{n}\bigg)=\LM(x)-
\bigg(x-\left\lceil\frac{x}{p}\right\rceil\bigg)\log_p(n)\quad(x\in\Z_p, n\in\N,p\nmid n),
\end{equation}
where $\left\lceil\frac{x}{p}\right\rceil$ is defined as the $p$-adic
limit of $\left\lceil\frac{x_n}{p}\right\rceil$ as $x_n\to x$ in $\N_0$. (This function
is known as Dwork's shift map and it is also written by $x\mapsto x'$.)

Diamond \cite{Dia} defined his $p$-adic
analogue of the classical $\LG$-function, which we will write $\LD$,
by the Volkenborn integral
\begin{equation}
\label{DiamondDefinition}
\LD(x):=\int_{\Z_p}(x+t)\big(\log_p(x+t)-1\big)dt
\qquad(x\in\C_p\setminus\Z_p).
\end{equation}
Diamond showed that his function is locally analytic on $x\in\C_p\setminus\Z_p$
and that it satisfies the difference
equation \cite[Theorem 5]{Dia}
\begin{equation}
\label{DiamondDiff}
\LD(x+1)-\LD(x)=\log_p(x)\qquad\qquad(x\in\C_p\setminus\Z_p),
\end{equation}
in analogy to \eqref{ClassicDiff}, and the reflection formula \cite[Theorem 8]{Dia}
\begin{equation}
\label{diamondRef}
\LD(1-x)+\LD(x)=0 \qquad\qquad (x\in\C_p\setminus\Z_p).
\end{equation}
Diamond's function also satisfies the distribution formula 
\cite[Theorem 7]{Dia} \cite[Theorem 60.2.(iii)]{Sch1}
\begin{equation}
\label{distLD}
\sum_{0\le j<n}\LD\bigg(\frac{x+j}{n}\bigg)=\LD(x)-
\bigg(x-\frac12\bigg)\log_p(n)\quad(x\in\C_p\setminus\Z_p, n\in\N).
\end{equation}

The functions of Morita and Diamond are connected
by the distribution formula \cite[Prop.~11.5.17.(4)]{Coh}
\begin{equation}
\label{distLMLD}
\sum_{\substack{0\leq j<n\\
p\nmid(x+j)}}\LD\bigg(\frac{x+j}{n}\bigg)=\LM(x)-
\bigg(x-\left\lceil\frac{x}{p}\right\rceil\bigg)\log_p(n)\qquad(x\in\Z_p, p\!\mid\!n).
\end{equation}

Comparing formulas \eqref{VolkenMor}
and \eqref{DiamondDefinition}, we notice that
$\LM$ and $\LD$ have very similar expressions
involving a Volkenborn integral.
Equations \eqref{MorRef} and \eqref{diamondRef}
show that $\LM$ and $\LD$
have identical reflection formulas, and 
\eqref{MorDiff} and \eqref{DiamondDiff} show
that these functions satisfy similar
difference equations. Equation \eqref{distLMLD} hints at a
connection between these two functions.
Also, the domains of $\LM$ and $\LD$
are disjoint and complementary in $\C_p$. Finally,
we mention that if we could extend
$\LD$ to $\C_p$, then the difference equation
would force $\LD$ to be discontinuous
at either the positive integers or the negative integers. Since
both these sets are dense in $\Z_p$, $\LD$ cannot be
extended continuously to any point of $\Z_p$. In particular,
we do not obtain a continuous function on $\C_p$ by extending the domain of $\LD$ 
defining $\LD(x):=\LM(x)$ for $x\in\Z_p$.

In 2005, Tomokazu Kashio \cite{Ka} defined a $p$-adic $\LG$-function which combines
these of Morita and Diamond. His definition is actually very general and here we work with its
simplest case. We take the approach due to Diamond, that is, we work with locally analytic 
functions and with the Volkenborn integral. It is worth mentioning that Kashio's definition
involves the Volkenborn integral but without mentioning it.

Kashio's definition, in terms of the Volkenborn integral, is as follows. Define
 $\LP:\C_p\to\C_p$ by the Volkenborn integral
$$
\LP(x):=\int_{\Z_p}
(x+t)\big(\log_p(x+t)-1\big)\chi(x+t)dt,
$$
where $\chi$ is the characteristic function of the
complement of the open unit ball
$B(0;1^-)^\mathsf{c}:=\big\{x\in\C_p\,\big|\,\absp{x}\geq1\big\}$.
This function will be proved to be locally analytic on $\C_p$,
and to satisfy the difference equation (Proposition \ref{PropDifLP})
$$
\LP(x+1) - \LP(x)=\chi(x)\log_p(x)\qquad (x\in\C_p),
$$
as well as the reflection formula (Proposition \ref{PropRefLP})
$$
\LP(1-x)+ \LP(x)=0\qquad\qquad (x\in\C_p).
$$

On certain sub-domains of $\C_p$, $\LP$ coincides with Morita's
function, on other with Diamond's. Namely, in \S 5 we will show that
\begin{align}
\notag
\LP(x)&=\LM(x)\qquad\qquad(x\in\Z_p\text{ or }\absp{x}<1),\\
\notag
\LP(x)&=\LD(x)\qquad\qquad(\absp{x}>1).
\end{align}
For $\absp{x}=1$ the relation between $\LP$ and Diamond's function
is more subtle (see Propositions \ref{PropositionLPLD}
and \ref{CorLPLM}).

In \S4 we will prove a distribution formula for $\LP$
(Proposition \ref{theo.dist.general1}). 
Let $x\in\C_p$, and fix $n=mp^r\in\N$\,.
We define next a finite  sequence $x_i\in\C_p$
of length at most $r+1$. Write $x_0:=x$. 
If $x_0\notin\zcp$, or
$x_0\in\zcp$ but $\ol{x_0}\notin\F_p$, then the sequence stops. 
If $x_0\in\zcp$ and $\ol{x_0}\in\F_p$, then define
$$
x_1:=\frac{x_0+p-\ell(x_0)}{p}.
$$
We repeat the above process: if $x_j\notin\zcp$, or
$x_j\in\zcp$ but $\ol{x_j}\notin\F_p$, then the sequence stops. 
If $x_j\in\zcp$ and $\ol{x_j}\in\F_p$, then
$$
x_{j+1}:=\frac{x_j+p-\ell(x_j)}{p}.
$$
Let $s$ the least non negative integer, or $+\infty$,  such that $x_s\notin\zcp$, or 
$x_s\in\zcp$ but $\ol{x_s}\notin\F_p$, and set $\omega=\min(r,s)$.
Define our sequence to be $x_0,\dots,x_\omega$.
Notice that the length of the sequence is $\omega+1\le r+1$.
Finally, define $\RP:\C_p\to\C_p$ by
$$
\RP(x):=
\begin{cases}\displaystyle
\medskip
x-\frac{x}{p}+\frac{\al}p-\left\lceil\frac{\al}{p}\right\rceil
\quad &\text{if } 
\absp{x-\al}<1\text{ for some } \al\in\Z,\\
\medskip
x-\frac{1}{2}  \quad &\text{otherwise,}
\end{cases}
$$
where $\lceil c\rceil$ is the usual integer ceiling function for $c\in\Q$.
Our distribution formula is
\begin{equation}
\label{distLDintro}
\sum_{k=0}^{n-1}\LP\bigg(\frac{x+k}n\bigg)=
\sum_{j=0}^{\omega}\LP(x_j)\ - \ 
\log_p(n)\sum_{j=0}^{\omega}\RP(x_j).
\end{equation}

As we shall show in \S 5, the distribution formulas \eqref{distLM} and \eqref{distLD}
are now special cases of \eqref{distLDintro} for $x\in\Z_p$ and $\absp{x}>1$, respectively.
Also, we will see that formula \eqref{distLDintro} may be
viewed as a generalization of the restricted distribution formula \eqref{distLMLD},
in the sense that to be able to omit the restriction in sum in the left hand of \eqref{distLMLD},
then there must appear other terms in its right hand. 

The author wishes to thank the anonymous referee for pointing out the work of Kashio,
and for many suggestions and corrections to the original manuscript.

\section{Locally analytic functions and the Volkenborn integral}

Let $D=\{x\in\C_p\,\big|\,\absp{x-y}<r\}$ be an open ball in 
$\C_p$ with center $y\in D$
and with positive radius $r$.
We will call a function $f:D\to\C_p$
analytic on $D$ if $f$ can be represented by a power series
\begin{equation}
\notag
f(x)=\sum_{n=0}^{\infty}a_n(x-y)^n
\end{equation}
convergent for all $x\in D$, where $a_n\in\C_p$ for all $n\in\N_0$.
Now, let $A$ be any subset of $\C_p$. We will  call a function
$f:A\to\C_p$ locally analytic on $A$ if
for each $a\in A$ there is an open ball $D\subset A$ with positive radius,
that contains $a$, such that $f$ is analytic on $D$. It is easily seen that
we may replace the word open for the word closed in this definition.

The following result, due to Diamond \cite{Dia}, will allow us to define
Kashio's $p$-adic $\log\GC$-function in terms of the Volkenborn integral.
\begin{Proposition}
\label{PropDiamond}
Let $f:\C_p\to\C_p$ be a locally analytic function on $\C_p$.
 Then, for all $b\in\N$, the limit
\begin{equation}
\label{VolkenGeneral}
F(x):=\lim_{n\to\infty}\frac1{bp^n}\sum_{j=0}^{bp^n-1}f(x+j)
\end{equation}
exists and is independent of $b$. Moreover, $F(x)$ defines a locally
analytic function on $\C_p$, and we have the identity
\begin{equation}
\label{VolkenGenDif}
F'(x)=\lim_{n\to\infty}\frac1{bp^n}\sum_{j=0}^{bp^n-1}f'(x+j).
\end{equation}
\end{Proposition}
\begin{proof}
The existence of \eqref{VolkenGeneral} is a special case of
\cite[p.~324, Corollary]{Dia}, and the identity \eqref{VolkenGenDif}
follows immediately from \cite[Theorem 3]{Dia}.
\end{proof}

If $g:\Z_p\to\C_p$ we say that
$g$ is Volkenborn integrable if the limit
$$
\int_{\Z_p} g(t)dt:=
\lim_{n\to\infty}\frac1{p^n}\sum_{j=0}^{p^n-1} g(j)
$$
exists, and we will call it the Volkenborn integral of $g$  \cite[$\S$V.5]{Rob} \cite[$\S$55]{Sch1}.
We can now restate Proposition \ref{PropDiamond} using the Volkenborn integral.

\begin{Lemma}
\label{LemmaVolken}
Let $f:\C_p\to\C_p$ be a locally analytic function on $\C_p$.
 Then the Volkenborn integral
\begin{equation}
\label{VolkenGeneral2}
F(x):=\int_{\Z_p}f(x+t)dt
\end{equation}
exists and $F(x)$ defines a locally
analytic function on $\C_p$. Moreover, we can differentiate
under the integral sign, that is
\begin{equation}
\notag 
F'(x)=\int_{\Z_p}f'(x+t)dt.
\end{equation}
\end{Lemma}
\begin{proof}
This follows immediately from the definition of the
Volkenborn integral,
letting $b=1$ in Proposition \ref{PropDiamond}.
\end{proof}

\begin{Remark}
The Volkenborn integral is usually defined for strictly
differentiable functions \cite[\S V.1.1]{Rob} \cite[\S27]{Sch1}.
Let $X$ be any non empty subset of $\C_p$ with no isolated points
and let $f:X \to \C_p$. We say that $f$ is strictly differentiable
 at a point $a\in X$ if
\begin{equation}
\notag
\lim_{(x,y)\to(a,a)}\frac{f(x)-  f(y)}{x-y}
\end{equation}
exists, where we take the limit over $x,y \in X$ such that $x \not = y$.
We say that $f$ is strictly
differentiable on $X$, or that $f\in C^1(X)$, if $f$ is strictly differentiable
for all $a\in X$. If $f\in C^1(\Z_p)$, then the Volkenborn
integral
$$
\label{Volkenborn}
\int_{\Z_p} f(t)dt:=
\lim_{n\to\infty}\frac1{p^n}\sum_{j=0}^{p^n-1} f(j)
$$
of $f$ exists \cite[\S V.5.1]{Rob} \cite[\S55]{Sch1}. All the properties
of the Volkenborn integral
that we will mention from \cite{Cofr}, \cite{Rob} and \cite{Sch1} are proved there for strictly
differentiable functions.
Since any locally analytic function on an open set $X\subset\C_p$
is  strictly differentiable on $X$ \cite[Corollary 29.11]{Sch1},
these properties hold for locally analytic functions on $X$.
\end{Remark}

Perhaps the simplest non trivial property of a function $F$ defined by
\eqref{VolkenGeneral2} is that it satisfies the difference equation
\cite[Theorem 4]{Dia} \cite[Prop.~55.5]{Sch1}
\begin{equation}
\label{VolkenbornDiff}
F(x+1)-F(x)=f'(x)
\qquad\qquad (x\in\C_p).
\end{equation}

We will also need the following ``distribution'' and ``integration by parts'' formulas.

\begin{Lemma}
\label{LemmaDist}
Let $f:\C_p\to\C_p$ be a locally analytic function on $\C_p$. Then for all
$N\in\N$
\begin{equation}
\notag
\int_{\Z_p}f(t)dt = \frac1N\sum_{k=0}^{N-1}\int_{\Z_p}f(k+Nt)dt.
\end{equation}
\end{Lemma}
\begin{proof}
See \cite[\S55]{Sch1}.
\end{proof}

\begin{Lemma}
\label{LemmaRaabe}
Let $f$ and $F$ be related as in Lemma \ref{LemmaVolken}. Then we
have the identity
\begin{equation}
\label{PreRaabe}
\int_{\Z_p}F(x+t)dt = F(x)+(x-1)F'(x)- \int_{\Z_p}
(x+t)f'(x+t)dt,
\end{equation}
valid for all $x\in\C_p$.
\end{Lemma}
\begin{proof}
See \cite[Lemma 2.2]{Cofr}.
\end{proof}


\section{Kashio's $p$-adic log-gamma function}

Let $\varphi_p:\C_p\to\C_p$ be the function defined by
\begin{equation}
\label{phiDef1}
\varphi_p(x):=
\begin{cases}\displaystyle
 0 \quad &\text{if }\absp{x}<1, \\
 x\log_p(x)-x \quad &\text{if } \absp{x}\geq1.
\end{cases}
\end{equation}
If we call $\chi$ the characteristic function of the set
$B(0;1^-)^\mathsf{c}=\big\{x\in\C_p\,\big|\,\absp{x}\geq1\big\}$,
we can write
\begin{equation}
\notag
\varphi_p(x)=x\big(\log_p(x)-1\big)\chi(x).
\end{equation}
Since the open ball $B(0;1^-)$
is also closed, its complement in $\C_p$ is open, so in
\eqref{phiDef1} we have defined $\varphi_p$ by its restriction to disjoint open
sets. Now, the null function is trivially analytic on $\C_p$, and so is the
identity function. Also, $\log_p(x)$ is locally
analytic on $B(0;1^-)^\mathsf{c}$, in fact on $\C_p^*$\,
\cite[Prop.~45.7]{Sch1}. Thus
the function $x\log_p(x)-x$ is also locally analytic on
the open set $B(0;1^-)^\mathsf{c}$. Hence,
the function $\varphi_p$ is locally analytic on $\C_p$. Therefore,
by Lemma \ref{LemmaVolken}, the following definition makes sense.

\begin{Def}
\label{DefLP}
With notation as above,
define the 
function \nolinebreak$\LP:\C_p\to\C_p$ by the Volkenborn
integral
\begin{equation}
\notag 
\LP(x):=\int_{\Z_p}\varphi_p(x+t)dt.
\end{equation}
\end{Def}

Hence, we can write
$$
\LP(x)=\int_{\Z_p}
(x+t)\big(\log_p(x+t)-1\big)\chi(x+t)dt,
$$
and by Lemma \ref{LemmaVolken}, $\LP$ is locally analytic on $\C_p$.

\begin{Remark} 
This is Kashio's $p$-adic $\LG$-function $L\Gamma_{p,1}(x,(1))$ 
(see \cite[eq.~5.12]{Ka}). His construction 
is much more general. He works with a multiple $p$-adic Hurwitz zeta-function, and he defines 
his multiple  $p$-adic $\LG$-function by means of the derivative at zero of this $p$-adic
Hurwitz zeta-function, as in the complex case. 
\end{Remark}

The simplest property of $\LP$ is its
difference equation.
\begin{Proposition}
\label{PropDifLP}
For all $x\in\C_p$ we have the difference equation
\begin{equation}
\notag
\LP(x+1) - \LP(x)=\chi(x)\log_p(x).
\end{equation}
\end{Proposition}
\begin{proof}
This follows from \eqref{VolkenbornDiff}, noticing that
\begin{equation}
\label{Derphi}
\ {\varphi_p}'(x)=\chi(x)\log_p(x),
\end{equation}
where $\chi$ is the characteristic function of the set
$B(0;1^-)^\mathsf{c}=\big\{x\in\C_p\,\big|\,\absp{x}\geq1\big\}$.
\end{proof}
Kashio proved the above formula in \cite[Lemma 5.5]{Ka} but with a mistake. He claims that
$\LP(x+1) - \LP(x)=\log_p(x)$, that is, he omits the factor $\chi(x)$.

The function $\LP$ satisfies a Raabe-type formula and
a characterization theorem similar to the ones satisfied by 
Diamond's $\LD$ and Morita's $\LM$ \cite[p.~364]{Cofr}.
\begin{Theorem}
\label{TheoRaabeLP}
The function $\LP$ satisfies
\begin{equation}
\label{LPRaabe}
\int_{\Z_p}\LP(x+t)dt = (x-1)
\LP'(x)-\rp(x)\qquad(x\in\C_p),
\end{equation}
where $\rp:\C_p\to\C_p$ is defined by the Volkenborn integral
\begin{equation}
\label{RPDef}
\rp(x):=\int_{\Z_p}(x+t)\chi(x+t)dt.
\end{equation}
Moreover, $\LP$ is
the unique locally analytic function
$f:\C_p\to\C_p$ satisfying the difference equation
\begin{equation}
\notag
f(x+1) - f(x)=\chi(x)\log_p(x)
\end{equation}
 and the Volkenborn integro-differential equation
\begin{equation}
\notag
\int_{\Z_p}f(x+t)dt = (x-1)f'(x)-\rp(x).
\end{equation}
\end{Theorem}
\begin{proof}
First we prove formula \eqref{LPRaabe}.
Using \eqref{PreRaabe} and \eqref{Derphi} we have
\begin{align}
\notag
\int_{\Z_p}\LP(x+t)dt &=\LP(x) +(x-1)\LP'(x)\\
\notag &\qquad  -\int_{\Z_p}(x+t)\log_p(x+t)\chi(x+t)dt\\
\notag &=\LP(x) +(x-1)\LP'(x)\\
\notag &\qquad-\LP(x)-\int_{\Z_p}(x+t)\chi(x+t)dt\\
\notag &=(x-1)\LP'(x)-\rp(x).
\end{align}

The uniqueness claim is proved exactly as in \cite{Cofr}.

\end{proof}

There is a reflection formula for $\LP$.
\begin{Proposition}
\label{PropRefLP}
For all $x\in\C_p$ we have the reflection formula
\begin{equation}
\notag
\LP(1-x) + \LP(x)=0.
\end{equation}
\end{Proposition}
\begin{proof}
Follows exactly as in \cite[Prop.~2.5]{Cofr}.
\end{proof}

The function $\rp$ defined by \eqref{RPDef} can be computed explicitly.

\begin{Proposition}
\label{PropValueRP}
For $x\in\C_p$ we have 
$$
\rp(x)=\RP(x):=
\begin{cases}\displaystyle
\medskip
x-\frac{x}{p}+\frac{\al}p-\left\lceil\frac{\al}{p}\right\rceil
\quad &\text{if } 
\absp{x-\al}<1\text{ for some } \al\in\Z,\\
\medskip
x-\frac{1}{2}  \quad &\text{otherwise,}
\end{cases}
$$
where $\lceil c\rceil$ is the usual integer ceiling function for $c\in\Q$.
\end{Proposition}
\begin{Remark}

One easily checks that $\absp{x-\al}<1$ implies that $\absp{x}\le1$, 
and that $\frac{\al}p-\left\lceil\frac{\al}{p}\right\rceil$
only depends on $\al$ modulo $p$.
\end{Remark}
\begin{proof}
We begin the proof  with the easy case, which is when $\absp{x}>1$.
If $\absp{t}\le1$ then
$\absp{x+t}=\absp{x}>1$. Thus, $\chi(x+j)=1$
for all $j\in\N_0$ and
\begin{align}
\notag
\rp(x)&=\int_{\Z_p}(x+t)\chi(x+t)dt
=\lim_{n\to\infty}\frac1{p^n}
\sum_{j=0}^{p^n-1}(x+j)\chi(x+j)\\
\notag &=\lim_{n\to\infty}\frac1{p^n}
\sum_{j=0}^{p^n-1}x\;+\;
\lim_{n\to\infty}\frac1{p^n}
\sum_{j=0}^{p^n-1}j
=x+\lim_{n\to\infty}\frac{p^n-1}{2}=x-\frac12.
\end{align}

Now, suppose that $\absp{x}\leq1$. Then
\begin{align}
\notag
\rp(x)&=\int_{\Z_p}(x+t)\chi(x+t)dt\\
\notag &=\lim_{n\to\infty}\frac1{p^n}
\sum_{0\le j<p^n}(x+j)\chi(x+j)\\
\notag &=\lim_{n\to\infty}\frac1{p^n}
\sum_{\substack{0\le j<p^n\\
\absp{x+j}=1}}(x+j)\\
\notag &=\lim_{n\to\infty}\frac1{p^n}
\sum_{0\le j<p^n}(x+j)\,
-\lim_{n\to\infty}\frac1{p^n}
\sum_{\substack{0\le j<p^n\\
\absp{x+j}<1}}(x+j)\\
\label{temp1} &=x-\frac12\,
-\lim_{n\to\infty}\frac1{p^n}
\sum_{\substack{0\leq j<p^n\\
\absp{x+j}<1}}(x+j).
\end{align}
In the last sum above, if there is no $j$ such that $\absp{x+j}<1$, 
this sum is 0. In this case we also have $\rp(x)=x-1/2$.
The remaining case is when
$\absp{x}\le1\text{ and }
\absp{x-\al}<1$ for some integer $\al$, which we
may choose to satisfy $0\leq\al\leq p-1$.
Then, in the sum in \eqref{temp1}, the condition $\absp{x+j}<1$
is equivalent to $\absp{\al+j}<1$.
Since $\al+j\in\N_0$, this is equivalent to the simpler
condition $p\mid(\al+j)$. Hence we have
\begin{align}
\notag
\sum_{\substack{0\leq j<p^n\\
\absp{x+j}<1}}(x+j) &=\sum_{\substack{0\leq j<p^n\\
p\mid(\al+j)}}(x-\al+\al+j)=\sum_{\substack{\al\leq i<\al+p^n\\
p\mid i}}(x-\al+i)\\
\notag &=\sum_{\frac\al{p}\leq j<\frac\al{p}+p^{n-1}}(x-\al+pj)
=\sum_{\left\lceil\frac{\al}{p}\right\rceil\le j <
\left\lceil\frac{\al}{p}\right\rceil+p^{n-1}}(x-\al+pj)\\
\notag &=\sum_{0\leq i<p^{n-1}}(x-\al+pi+p\left\lceil\frac{\al}{p}\right\rceil)
=p^{n-1}(x-\al)+p^n\left\lceil\frac{\al}{p}\right\rceil+
p\sum_{0\leq i<p^{n-1}}i.
\end{align}
Replacing this in \eqref{temp1} we obtain
\begin{align}
\notag
\rp(x)&=x-\frac12\,
-\lim_{n\to\infty}\frac1{p^n}
\bigg(p^{n-1}(x-\al)+p^n\left\lceil\frac{\al}{p}\right\rceil+
p\sum_{0\leq i<p^{n-1}}i\bigg)\\
\notag &=x-\frac12-\frac{x}p+\frac{\al}p-
\left\lceil\frac{\al}{p}\right\rceil-
\lim_{n\to\infty}\frac1{p^{n-1}}\sum_{0\leq i<p^{n-1}}i\\
\notag &=x-\frac{x}{p}+\frac{\al}p-\left\lceil\frac{\al}{p}\right\rceil.
\end{align}
\end{proof}

\begin{Corollary}
\label{CorValueRP}
For $x\in\Q_p$ we have
$$
\rp(x)=\RP(x)=
\begin{cases}\displaystyle
\medskip
x-\left\lceil\frac{x}{p}\right\rceil
\quad &\text{if } 
x\in\Z_p,\\
\medskip
x-\frac{1}{2}  \quad &\text{otherwise,}
\end{cases}
$$
where $\left\lceil\frac{x}{p}\right\rceil$ is defined as the $p$-adic
limit of $\left\lceil\frac{x_n}{p}\right\rceil$ as $x_n\to x$ in $\N_0$.
\end{Corollary}
\begin{proof}
It is easily seen that the limit of $\left\lceil\frac{x_n}{p}\right\rceil$ exists
and that $\left\lceil\frac{y}{p}\right\rceil=\frac{y}{p}$ if $y\in p\Z_p$. Now, 
if $x\in\Z_p$, then $x=\al+y$ where $\absp{x-\al}<1$ is such that 
$0\le\al\le p-1$ and where $y\in p\Z_p$. Then
\begin{equation}
\notag
x-\frac{x}{p}+\frac{\al}p-\left\lceil\frac{\al}{p}\right\rceil=
x-\frac{y}{p}-\left\lceil\frac{x-y}{p}\right\rceil=
x-\frac{y}{p}-\left\lceil\frac{x}{p}\right\rceil+\left\lceil\frac{y}{p}\right\rceil=
x-\left\lceil\frac{x}{p}\right\rceil.
\end{equation}
\end{proof}

\begin{Remark}
The function $\rp(x)$ is a zeta value. More precisely, $\rp(x)=-\zeta_{p,1}(0,(1),x)$,
where $\zeta_{p,1}(s,(1),x)$ is Kashio's Hurwitz zeta function.
In the complex case, the value 
at $s=0$ of the Hurwitz zeta-function is $\zeta(0,x)=-x+1/2$,
 which is in agreement with the $p$-adic
case when $x\in\C_p\setminus\Z_p$.
\end{Remark}


\section{The distribution formula}

Let $n=mp^r$ with $p\nmid m$, and let $g:\C_p\to\C_p$ be any $p$-adic function. 
Suppose we want to compute the sum
$\sum_{k=0}^{n-1}g\bigg(\frac{x+k}{n}\bigg)$. Then we can write
\begin{align}
\notag
\sum_{k=0}^{n-1}g\bigg(\frac{x+k}{n}\bigg)&=
\sum_{k=0}^{mp^r-1}g\bigg(\frac{x+k}{mp^r}\bigg)=
\sum_{i=0}^{m-1} \ \sum_{j=ip^r}^{p^r\!-\!1\!+ip^r}g\bigg(\frac{x+j}{mp^r}\bigg)\\
\notag&=\sum_{i=0}^{m-1}\sum_{j=0}^{p^r-1}g\bigg(\frac{x+j+ip^r}{mp^r}\bigg)=
\sum_{j=0}^{p^r-1}\sum_{i=0}^{m-1}g\bigg(\frac{\frac{x+j}{p^r}+i}{m}\bigg).
\end{align}
Let $y=\frac{x+j}{p^r}$. If we can compute
$\sum_{i=0}^{m-1}g\bigg(\frac{y+i}{m}\bigg)$, then we reduce
the comuputation of the original sum to the case $n=p^r$. This can be done when 
$g=\LP$.

\begin{Lemma}
Let $p\nmid m$. Then, for all $y\in\C_p$,
\begin{equation}
\notag
\sum_{i=0}^{m-1}\LP\bigg(\frac{y+i}m\bigg)=
\LP(y)-\log_p(m)\RP(y).
\end{equation}
\end{Lemma}
\begin{proof}
Since $\absp{m}=1$,
\begin{align}
\notag
\sum_{i=0}^{m-1}\LP&\bigg(\frac{y+i}m\bigg)=\sum_{i=0}^{m-1}\int_{\Z_p}
\bigg(\frac{y+i}m+t\bigg)\big(\log_p\bigg(\frac{y+i}m+t\bigg)-1\big)
\chi\bigg(\frac{y+i}m+t\bigg)dt\\
\notag&=\frac1m\sum_{i=0}^{m-1}
\int_{\Z_p}(y+i+mt)\big(\log_p(y+i+mt)-\log_p(m)-1\big)\chi(y+i+mt)dt.
\end{align}
Using Lemma \ref{LemmaDist}, we conclude that
\begin{align}
\notag
\sum_{i=0}^{m-1}\LP&\bigg(\frac{y+i}m\bigg)=
\int_{\Z_p}(y+t)\big(\log_p(y+t)-\log_p(m)-1\big)\chi(y+t)dt\\
\notag&=\int_{\Z_p}(y+t)\big(\log_p(y+t)-1\big)\chi(y+t)dt
-\log_p(m)\int_{\Z_p}(y+t)\chi(y+t)dt\\
\notag&=\LP(y)-\log_p(m)\RP(y).
\end{align}
\end{proof}

By the above comments we obtain
\begin{equation}
\label{casepr}
\sum_{k=0}^{n-1}\LP\bigg(\frac{x+k}{n}\bigg)=
\sum_{j=0}^{p^r-1}\LP\bigg(\frac{x+j}{p^r}\bigg)-
\log_p(n)\sum_{j=0}^{p^r-1}\RP\bigg(\frac{x+j}{p^r}\bigg).
\end{equation}

Now we generalize a little bit to compute the two sums in the right hand of 
\eqref{casepr} working with only one function.

Let $f:\C_p\setminus\{0\}\to\C_p$, i.e., a function possibly not defined at 0.
Suppose $f$ is locally analytic and $f(px)=f(x)$ for all 
$x\in\C_p\setminus\{0\}$. Then clearly
\begin{equation}
\label{fpk}
f(p^kx)=f(x)
\end{equation} 
for all $k\in\Z$ and $x\in\C_p\setminus\{0\}$.   
Let $F$ be the function defined by
$$
F(x)=\int_{\Z_p}(x+t)f(x+t)\chi(x+t)dt.
$$
Then $F$ is a locally analytic function over $\C_p$ (same proof  as for $\LP$).

Notice that if $f(x)=1$ then
$F(x)=\RP(x)$, and  if $f(x)=\log_p(x)-1$ then $F(x)=\LP(x)$.
Thus, we are going to compute $\sum_{j=0}^{p^r-1}F\bigg(\frac{x+j}{p^r}\bigg)$,
and this would give us the value of the right hand of \eqref{casepr}. We
will do this in two complementary disjoint cases. 
First we need some properties of $\chi$.

\begin{Lemma}
\label{lemmay+z}
If $\absp{z}<1$, then $\chi(y)=\chi(y+z)$ for all $y\in\C_p$.
\end{Lemma}
\begin{proof}
If $\chi(y)=1$ then $\absp{y}\ge1$. In this case $\absp{y+z}=\absp{y}\ge1$,
hence $\chi(y+z)=1$. If $\chi(y)=0$ then $\absp{y}<1$. In this case 
$\absp{y+z}\le\max\{\absp{y},\absp{z}\}<1$,
hence $\chi(y+z)=0$. 
\end{proof}

\begin{Lemma}
\label{lemmawp}
Let $W_p$ be the set containing all $x\in\C_p$ such that  
$x\notin\zcp$, or $x\in\zcp$ and $\overline{x}\notin\F_p$. Then
$W_p$ is $\Z_p$-invariant, meaning that, if $x\in W_p$ and $t\in\Z_p$,
then $x+t\in W_p$.
\end{Lemma}
\begin{proof}
By cases. First, $x\notin\zcp$ means that $\absp{x}>1$. If $t\in\Z_p$, then
$\absp{t}\le1$, so that $\absp{x+t}=\absp{x}>1$, and then $x+t\notin\zcp$.
Now, let $x\in\zcp$ such that $\overline{x}\notin\F_p$ and let $t\in\Z_p$. 
Since $\overline{t}\in\F_p$, if $\overline{x+t}\in\F_p$, then $\overline{x+t}-\overline{t}=
\overline{x}+\overline{t-t}=\overline{x}\in\F_p$, a contradiction.
\end{proof}

\begin{Corollary}
\label{corxrjt}
Let $x\in W_p$ and $t\in\Z_p$. Then $\chi(x+t)=1$, and 
$\chi((x+j)/p^r+t)=1$ for all $r\in\N$ and $j\in\N_0$.
\end{Corollary}
\begin{proof}
First notice that if $y\in W_p$ then $\absp{y}\ge1$; 
if not, $\absp{y}<1$ and this would imply 
$\overline{y}=\overline{0}\in\F_p$. Let $t\in\Z_p$. 
If $x\in W_p$, by Lemma \ref{lemmawp}, $x+t\in W_p$ so that 
$\absp{x+t}\ge1$ and this gives $\chi(x+t)=1$. Also, if 
$j\in\N_0$, $x+j\in W_p$, so that $\absp{x+j}\ge1$. Then, for $r\in\N$,
$\absp{(x+j)/p^r}=\absp{(x+j)}p^r>1$. In 
particular $(x+j)/p^r\in W_p$, so that $(x+j)/p^r+t\in W_p$,
and then $\absp{(x+j)/p^r+t}\ge1$, i.e., $\chi((x+j)/p^r+t)=1$.
\end{proof}

Now we compute the mentioned sums involving $F$.
Recall that for $x\in\zcp$ such that $\ol{x}\in\F_p$ we let $\ell(x)$
be the unique natural number satisfying $1\le\ell(x)\le p$ and
$\ol{x}=\ol{\ell(x)}$.

\begin{Lemma}
\label{lemmayesfp}
Let $r\in\N$ and let $x\in\zcp$ such that $\overline{x}\in\F_p$. Then
\begin{equation}
\notag
\sum_{k=0}^{p^r-1}F\bigg(\frac{x+k}{p^r}\bigg)=
F(x)+\sum_{j=0}^{p^{r-1}-1}F\bigg(\frac{x'+j}{p^{r-1}}\bigg),
\end{equation}
where $x'=(x+p-\ell(x))/p$.
\end{Lemma}
\begin{proof}
Since $\absp{x}\le1$, then $\absp{x+y}\le1$ for all $y\in\zcp$, and in particular,
$\absp{x+k}\le1$ for all $k\in\N_0$. Hence, we can write
\begin{equation}
\label{FcaseFp}
\sum_{0\leq k<p^r}F\bigg(\frac{x+k}{p^r}\bigg)=
\sum_{\substack{0\leq k<p^r\\ \absp{x+k}<1}}F\bigg(\frac{x+k}{p^r}\bigg)+
\sum_{\substack{0\leq k<p^r\\ \absp{x+k}=1}}F\bigg(\frac{x+k}{p^r}\bigg).
\end{equation}

We first compute the first sum in the right hand of \eqref{FcaseFp}. Notice
that $\absp{x+k}<1$  if and only if $\absp{\ell(x)+k}<1$, and this ocurrs
 if and only if $p\mid(\ell(x)+k)$. Hence
\begin{align}
\notag
\sum_{\substack{0\leq k<p^r\\ \absp{x+k}<1}}F\bigg(\frac{x+k}{p^r}\bigg)&=
\sum_{\substack{0\leq k<p^r\\ p\mid\ell(x)+k}}F\bigg(\frac{x+k}{p^r}\bigg)=
\sum_{\substack{p-\ell(x)\leq k<p^r-\ell(x)\\ p\mid\ell(x)+k}}F\bigg(\frac{x+k}{p^r}\bigg)\\
\notag&=\sum_{\substack{0\leq i\leq p^r-p\\ p\mid i}}F\bigg(\frac{x+p-\ell(x)+i}{p^r}\bigg)\\
\notag&=\sum_{0\leq j<p^{r-1}}F\bigg(\frac{x+p-\ell(x)+pj}{p^r}\bigg)\\
\notag&=\sum_{0\leq j<p^{r-1}}F\bigg(\frac{x'+j}{p^{r-1}}\bigg).
\end{align}

Now, recalling the definition of $F$, the second sum in the right hand of \eqref{FcaseFp} is
$$
\sum_{\substack{0\leq k<p^r\\ \absp{x+k}=1}}F\bigg(\frac{x+k}{p^r}\bigg)=
\sum_{\substack{0\leq k<p^r\\ \absp{x+k}=1}}\int_{\Z_p}
\bigg(\frac{x+k}{p^r}+t\bigg)f\bigg(\frac{x+k}{p^r}+t\bigg)
\chi\bigg(\frac{x+k}{p^r}+t\bigg)dt.
$$
Since $\absp{x+k}=1$ and $r\ge1$, then $\absp{(x+k)/p^r}=p^r>1$.
Since also $\absp{t}\le1$, then  $\absp{(x+k)/p^r+t}=p^r>1$, and we
deduce that $\chi((x+k)/p^r+t)=1$. Then, using \eqref{fpk},
\begin{align}
\notag
\sum_{\substack{0\leq k<p^r\\ \absp{x+k}=1}}F\bigg(\frac{x+k}{p^r}\bigg)&=
\frac1{p^r}\sum_{\substack{0\leq k<p^r\\ \absp{x+k}=1}}\int_{\Z_p}
(x+k+p^rt)f(x+k+p^rt)dt\\
\notag&=\frac1{p^r}\sum_{0\leq k<p^r}\int_{\Z_p}
(x+k+p^rt)f(x+k+p^rt)\chi(x+k)dt,
\end{align}
and using Lemma \ref{lemmay+z} with $y=x+k$ and $z=p^rt$, and
Lemma \ref{LemmaDist}, we finally obtain
\begin{align}
\notag
\sum_{\substack{0\leq k<p^r\\ \absp{x+k}=1}}F\bigg(\frac{x+k}{p^r}\bigg)&
=\frac1{p^r}\sum_{0\leq k<p^r}\int_{\Z_p}(x+k+p^rt)f(x+k+p^rt)\chi(x+k+p^rt)dt\\
\notag&=\int_{\Z_p}(x+t)f(x+t)\chi(x+t)dt\\
\notag&=F(x).
\end{align}
The lemma follows.
\end{proof}

\begin{Lemma}
\label{lemmanotfp}
Let $r\in\N$, and let $x\in\C_p$ such that $x\notin\zcp$, 
or $x\in\zcp$ and $\overline{x}\notin\F_p$. Then
\begin{equation}
\notag
\sum_{k=0}^{p^r-1}F\bigg(\frac{x+k}{p^r}\bigg)=F(x).
\end{equation}
\end{Lemma}
\begin{proof}
Using Corollary \ref{corxrjt}, equation \eqref{fpk} and Lemma \ref{LemmaDist}, we obtain
\begin{align}
\notag
\sum_{k=0}^{p^r-1}F\bigg(\frac{x+k}{p^r}\bigg)&=\sum_{k=0}^{p^r-1}\int_{\Z_p}
\bigg(\frac{x+k}{p^r}+t\bigg)f\bigg(\frac{x+k}{p^r}+t\bigg)dt\\
\notag&=\frac1{p^r}\sum_{k=0}^{p^r-1}\int_{\Z_p}(x+k+p^rt)f(x+k+p^rt)dt\\
\notag&=\int_{\Z_p}(x+t)f(x+t)dt\\
\notag&=F(x).
\end{align}
\end{proof}

Now, let $x\in\C_p$, and fix $n=mp^r\in\N$\,.
We define next a finite  sequence $x_i\in\C_p$
of length at most $r+1$. Write $x_0:=x$. 
If $x_0\notin\zcp$, or
$x_0\in\zcp$ but $\ol{x_0}\notin\F_p$, then the sequence stops. 
If $x_0\in\zcp$ and $\ol{x_0}\in\F_p$, then define
$$
x_1:=\frac{x_0+p-\ell(x_0)}{p}.
$$
We repeat the above process: if $x_j\notin\zcp$, or
$x_j\in\zcp$ but $\ol{x_j}\notin\F_p$, then the sequence stops. 
If $x_j\in\zcp$ and $\ol{x_j}\in\F_p$, then
$$
x_{j+1}:=\frac{x_j+p-\ell(x_j)}{p}.
$$
Let $s$ the least non negative integer, or $+\infty$,  such that $x_s\notin\zcp$, or 
$x_s\in\zcp$ but $\ol{x_s}\notin\F_p$, and set $\omega=\min(r,s)$.
Define our sequence to be $x_0,\dots,x_\omega$.
Notice that the length of the sequence is $\omega+1\le r+1$.

\begin{Proposition}
\label{mainprop}
With notation as above,
\begin{equation}
\notag
\sum_{k=0}^{p^r-1}F\bigg(\frac{x+k}{p^r}\bigg)=
\sum_{j=0}^{\omega}F(x_j).
\end{equation}
\end{Proposition}
\begin{proof}
Follows inductively applying Lemma \ref{lemmayesfp} and Lemma \ref{lemmanotfp}.
\end{proof}

As a corollary we prove our distribution formula for $\LP$.

\begin{Theorem}
\label{theo.dist.general1}
With notation as above,
\begin{equation}
\label{dist.general1}
\sum_{k=0}^{n-1}\LP\bigg(\frac{x+k}n\bigg)=
\sum_{j=0}^{\omega}\LP(x_j)\ - \ 
\log_p(n)\sum_{j=0}^{\omega}\RP(x_j).
\end{equation}

\end{Theorem}
\begin{proof}
Apply Proposition \ref{mainprop} for $F=\LP$ and $F=\RP$
in formula \eqref{casepr}.
\end{proof}


\section{Relation with the functions of Diamond and Morita}
We now take a look at the relation of $\LP$ with Diamond's 
and Morita's functions. 

Let us start with Diamond's \cite{Dia} function 
\begin{equation}
\notag
\LD(x):=\int_{\Z_p}(x+t)\big(\log_p(x+t)-1\big)dt
\qquad\qquad(x\in\C_p\setminus\Z_p).
\end{equation}

Recall that in Lemma \ref{lemmawp} we defined $W_p$ to be 
the set containing the $x\in\C_p$ such that  
$x\notin\zcp$, or $x\in\zcp$ and $\overline{x}\notin\F_p$. 
Obviously, $W_p\subset\C_p\setminus\Z_p$.
\begin{Proposition}
\label{PropositionLPLD}
For $x\in W_p$ we have
\begin{equation}
\notag
\LP(x)=\LD(x).
\end{equation}
\end{Proposition}
\begin{proof}
By Corollary \ref{corxrjt}, if $x\in W_p$ then $\chi(x+t)=1$ 
for all $t\in\Z_p$, and thus
$$
\LP(x)=\int_{\Z_p}(x+t)\big(\log_p(x+t)-1\big)dt
\qquad\qquad(x\in W_p),
$$
i.e., $\LP$ and $\LD$ are identical on $W_p$.
\end{proof}

Let us prove now that the distribution formula \eqref{distLD} 
restricted to $W_p$,
this is, for all $n\in\N$
\begin{equation}
\notag
\sum_{k=0}^{n-1}\LD\bigg(\frac{x+k}{n}\bigg)=\LD(x)-
\bigg(x-\frac12\bigg)\log_p(n)\qquad(x\in W_p),
\end{equation}
is a special case of Theorem \ref{theo.dist.general1}.
First, if $x\in W_p$, then by definition, we have that
$s=0$ in our sequence. Hence, $\omega=\min(r,s)=0$
and \eqref{dist.general1} becomes
\begin{equation}
\notag
\sum_{k=0}^{n-1}\LP\bigg(\frac{x+k}n\bigg)=
\LP(x)\ - \ \log_p(n)\RP(x),
\end{equation}
since $x_0=x$. Now, it is easily seen after Lemma
\ref{lemmawp} that if $x\in W_p$ then all the numbers $(x+k)/n$
are also in $W_p$. Since $\LP(x)=\LD(x)$ for $x\in W_p$, then
the equation above becomes
\begin{equation}
\notag
\sum_{k=0}^{n-1}\LD\bigg(\frac{x+k}n\bigg)=
\LD(x)\ - \ \log_p(n)\RP(x).
\end{equation}
Finally, using Proposition \ref{PropValueRP} for $x\in W_p$, we obtain
\begin{equation}
\notag
\sum_{k=0}^{n-1}\LD\bigg(\frac{x+k}n\bigg)=
\LD(x)\ - \ \log_p(n)\bigg(x-\frac12\bigg).
\end{equation}

We now consider the relation of $\LP$
with Morita's \cite{Mor} function $\LM$. Recall that $\LM$ is defined
by the Volkenborn integral
$$
\LM(x) =\int_{\Z_p}(x+t)\big(\log_p(x+t)-1\big)
\chi^{\phantom{-1}}_{\Z_p^*}\!(x+t)dt
\qquad (x\in\Z_p).
$$
The direct relation between $\LP$ and $\LM$
is easy since $\LM(x)$ is actually $\LP(x)$ restricted to $\Z_p$.
To see this, let us go back to the function $\varphi_p$ defined 
by \eqref{phiDef1}.
We have
\begin{equation}
\notag
\varphi_p(x)=
\begin{cases}\displaystyle
 0 \quad &\text{if }\absp{x}<1, \\
 x\log_p(x)-x \quad &\text{if } \absp{x}\geq1,
\end{cases}
\end{equation}
and if we restrict ourselves to $x\in\Z_p$, then
$$
\big\{x\in\Z_p\,\big|\,\absp{x}<1\big\}=p\Z_p \text{ \  and \ }
\big\{x\in\Z_p\,\big|\,\absp{x}\geq1\big\}=\Z_p^*.
$$
Hence we have
\begin{equation}
\notag
\LP(x) =\int_{\Z_p}\chi^{\phantom{-1}}_{\Z_p^*}\!(x+t)
(x+t)\big(\log_p(x+t)-1\big)dt
\qquad (x\in\Z_p),
\end{equation}
and this is exactly Morita's function $\LM$.

An important property of $\LM$ is that it has a
power series expansion around $0$, valid for all $x\in p\Z_p$.
Namely,
for all $x\in p\Z_p$ we have the identity  \cite[Lemma 58.1]{Sch1}
\begin{equation}
\label{MorSeries}
\LM(x)=\lambda_1x +\sum_{n=1}^{\infty}
\frac{(-1)^{n+1}\,\lambda_{n+1}}{n(n+1)}\ x^{n+1},
\end{equation}
where
$$
\lambda_1:=\int_{\Z_p}\chi^{\phantom{-1}}_{\Z_p^*}\!(t)
\log_p(t)dt,\quad
\lambda_{n+1}:=\int_{\Z_p}\chi^{\phantom{-1}}_{\Z_p^*}\!(t)\,
t^{-n}dt
\qquad(n\in\N).
$$
Moreover, the right side of \eqref{MorSeries} defines an analytic function
on the open unit ball  $B(0;1^-)\subset\C_p$ \cite[Lemma 58.2]{Sch1}.
Hence, we extend the domain of $\LM$ to $B(0;1^-)\cup\Z_p$ by
defining
\begin{equation}
\label{LMinB1}
\LM(x):=\lambda_1x +\sum_{n=1}^{\infty}
\frac{(-1)^{n+1}\,\lambda_{n+1}}{n(n+1)}\ x^{n+1}
\qquad(x\in B(0;1^-)).
\end{equation}
It can be shown that with this extended definition
$\LM$ is no longer the Iwasawa 
logarithm of a $p$-adic function.
\begin{Proposition}
\label{CorLPLM}
For all $x\in\Z_p\cup B(0;1^-)$ we have $\LP(x)=\LM(x)$.
\end{Proposition}
\begin{proof}
We already proved the equality on $\Z_p$. To prove the
equality on $B(0;1^-)$ first notice that both functions are
equal on $p\Z_p\subset B(0;1^-)$ and that $p\Z_p$ has
infinitely many accumulation points. Since $\LM(x)$ has
the power series expansion \eqref{LMinB1} on $B(0;1^-)$
and $\LP(x)$ is locally analytic on $\C_p$, then $\LP(x)$ 
must be equal to the right side of \eqref{LMinB1} on $B(0;1^-)$.
\end{proof}

Let us prove now that the distribution formula \eqref{distLM},
this is, for all $n\in\N$ not divisible by $p$
\begin{equation}
\notag
\sum_{k=0}^{n-1}\LM\bigg(\frac{x+k}{n}\bigg)=\LM(x)-
\bigg(x-\left\lceil\frac{x}{p}\right\rceil\bigg)\log_p(n)\qquad(x\in\Z_p),
\end{equation}
is a special case of Theorem \ref{theo.dist.general1}.
First, if $p\nmid n$, then by definition, we have that
$r=0$ in our sequence. Hence, $\omega=\min(r,s)=0$
and \eqref{dist.general1} becomes
\begin{equation}
\notag
\sum_{k=0}^{n-1}\LP\bigg(\frac{x+k}n\bigg)=
\LP(x)\ - \ \log_p(n)\RP(x),
\end{equation}
since $x_0=x$. Now, it is easily seen that if $p\nmid n$ and if $x\in\Z_p$ 
then all the numbers $(x+k)/n$
are also in $\Z_p$. Since $\LP(x)=\LM(x)$ for $x\in\Z_p$, then
the equation above becomes
\begin{equation}
\notag
\sum_{k=0}^{n-1}\LM\bigg(\frac{x+k}n\bigg)=
\LM(x)\ - \ \log_p(n)\RP(x).
\end{equation}
Finally, using Corollary \ref{CorValueRP} for $x\in\Z_p$, we obtain
\begin{equation}
\notag
\sum_{k=0}^{n-1}\LM\bigg(\frac{x+k}n\bigg)=
\LM(x)\ - \ \log_p(n)\bigg(x-\left\lceil\frac{x}{p}\right\rceil\bigg).
\end{equation}

We now take a look at the distribution formula \eqref{distLMLD}, that connects
the functions of Morita and Diamond. Let $x\in\Z_p$ and $n\in\N$ divisible by $p$. Then
\begin{equation}
\notag
\sum_{\substack{0\leq j<n\\
p\nmid(x+j)}}\LD\bigg(\frac{x+j}{n}\bigg)=\LM(x)-
\bigg(x-\left\lceil\frac{x}{p}\right\rceil\bigg)\log_p(n).
\end{equation}
By Propositions  \ref{PropositionLPLD} and \ref{CorLPLM},  and by Corollary \ref{CorValueRP},
this formula becomes
\begin{equation}
\label{distfinal}
\sum_{\substack{0\leq j<n\\
p\nmid(x+j)}}\LP\bigg(\frac{x+j}{n}\bigg)=\LP(x)-
\RP(x)\log_p(n)\qquad(x\in\Z_p, p\!\mid\!n).
\end{equation}
Formula \eqref{dist.general1} explains now the restriction in the sum in the left hand
of \eqref{distfinal}:
to be able to omit it then there must appear other terms in its right hand.

Finally, we mention that one can compute explicitly the expansions
in power series of $\LP(x)$ using those of $\LD(x)$ and $\LM(x)$, 
but for the sake of brevity, we omit the calculations here.


\bibliographystyle{amsplain}

\bibliography{references}

%
%
%
%
%

\end{document}